\documentclass[11pt, a4paper]{amsart}

\usepackage{amsfonts,amsmath,amssymb, amscd}

\newtheorem{theorem}{Theorem}[section]
\newtheorem{lemma}[theorem]{Lemma}
\newtheorem{definition}[theorem]{Definition}
\newtheorem{prop}[theorem]{Proposition}

\theoremstyle{definition}
\newtheorem{rem}[theorem]{Remark}
\newtheorem{exas}[theorem]{Examples}

\numberwithin{equation}{section}

\newcommand\eps{\varepsilon}
\newcommand\Frac{\operatorname{Frac}}

\renewcommand\AA{\mathcal A}
\newcommand\BB{\mathcal B}
\newcommand\UU{\mathcal U}
\newcommand\VV{\mathcal V}
\newcommand\ZZ{\mathcal Z}
\newcommand\mm{\mathfrak{m}}
\newcommand\ZZZ{\mathbb Z}

\DeclareMathOperator{\Ker}{Ker}
\DeclareMathOperator{\id}{id}
\DeclareMathOperator{\Id}{Id}
\DeclareMathOperator{\ab}{ab}

\title[Generic base algebras and universal comodule algebras]
{Generic base algebras and\\ universal comodule algebras for\\
some finite-dimensional Hopf algebras}

\author[Uma N. Iyer]
{Uma N. Iyer}
\address{Uma N. Iyer: 
Department of Mathematics and Computer Science, Bronx Community College,
2155 University Avenue, Bronx, New York 10453, USA}
\email{uma.iyer@bcc.cuny.edu}

\author[Christian Kassel]
{Christian Kassel}
\address{Christian Kassel: 
Institut de Recherche Math\'e\-ma\-tique Avanc\'ee,
CNRS \& Universit\'e de Strasbourg,
7 rue Ren\'{e} Descartes, 67084 Strasbourg, France}
\email{kassel@math.unistra.fr}



\begin{document}

\begin{abstract}
After recalling the definitions and the properties of the ge\-ner\-ic base algebra
and of the universal comodule algebra attached to a Hopf algebra 
by Aljadeff and the second-named author,
we determine these algebras for the Taft algebras, the Hopf algebras~$E(n)$
and certain monomial Hopf algebras.
\end{abstract}

\maketitle

\section*{Introduction}

In\,\cite{AK} Eli Aljadeff and the second-named author associated to any finite-dimensional Hopf algebra~$H$
an algebra~$\BB_H$ of rational fractions, which is a finitely generated smooth domain
of Krull dimension equal to the dimension of~$H$.
The algebra~$\BB_H$ is called the \emph{generic base algebra} associated to~$H$; it is the subalgebra of coinvariants
of the generic Hopf Galois extension~$\AA_H$ introduced in \emph{loc.\ cit.}
One can view~$\BB_H$ as the ``base space'' of a ``non-commutative fiber bundle'' 
whose fibers are ``forms'' of~$H$ (see also\,\cite{Ka0}).

The generic base algebra is known for very few Hopf algebras:
it has been described for finite group algebras in\,\cite{AHN} 
and for the four-dimensional Sweedler algebra in\,\cite{AK}.
The first objective of this paper is to determine~$\BB_H$ for other finite-dimensional Hopf algebras
such as the Taft algebras, the Hopf algebras~$E(n)$ and certain monomial Hopf algebras, 
all natural generalizations of the Sweedler algebra.
Our calculations are based on the properties of~$\BB_H$ established in~\cite{KM}.

A theory of polynomial identities for comodule algebras was also worked out in\,\cite{AK}. It leads naturally to a
\emph{universal comodule algebra}~$\UU_H$, 
the analogue of the ``relatively free algebra'' in the classical theory of polynomial identities. 
The subalgebra of coinvariants~$\VV_H$ of~$\UU_H$ maps injectively into~$\BB_H$. 
In the few known cases, the injection turns~$\BB_H$ into a localization of~$\VV_H$.
We show that this also holds for the Hopf algebras considered here.
Finally for the same Hopf algebras we also describe a suitable central localization
of~$\UU_H$ as a $\BB_H$-module.

The paper is organized as follows.
In Section~\ref{sec-PI} we recall the definition and the properties of the main objects under investigation, namely
the universal comodule algebra and the generic base algebra.
We also raise the questions of finding a presentation by generators and relations 
of the generic base algebra~$\BB_H$ and 
of deciding whether~$\BB_H$ is a localization of the subalgebra of coinvariants~$\VV_H$ of~$\UU_H$.

In Section\,\ref{Taft-def} we give complete answers to these questions for Taft algebras.
We give similar answers for the Hopf algebras~$E(n)$ in Section~\ref{sec-En}
and for certain monomial Hopf algebras in Section~\ref{sec-monomial}.

Appendix\,\ref{app-kG} is concerned with the group algebra case; 
we summarize the results of~\cite{AHN,KM} 
relevant to the present paper, and we compute the generic base algebra for certain finite groups.

\section{Polynomial identities}\label{sec-PI}

We fix a ground field~$k$ of characteristic zero. 
All vector spaces, all algebras considered in this paper are
defined over~$k$; similarly, all linear maps are supposed to be $k$-linear.
The symbol~$\otimes$ denotes the tensor product over~$k$.

\subsection{Hopf algebras and comodule algebras}\label{ssec-prel}

We refer to~\cite{Mo} for generalities on Hopf algebras and comodules algebras.
As is traditional, we denote the coproduct of a Hopf algebra by~$\Delta$, its counit by~$\eps$,
and its antipode by~$S$.
We also use a Heyneman-Sweedler-type notation  
\[
\Delta(x) = x_1 \otimes x_2
\]
for the image under~$\Delta$ of an element~$x$ of a Hopf algebra~$H$, and we write
\[
\Delta^{(2)}(x) = x_1 \otimes x_2 \otimes x_3
\]
for its image under the iterated coproduct 
$\Delta^{(2)} = (\Delta \otimes \id_H) \circ \Delta$.

Recall that a (right) $H$-\emph{comodule algebra} over a Hopf $k$-algebra~$H$
is an associative unital $k$-algebra~$A$ 
equipped with a right $H$-comodule structure whose (coassociative, counital) \emph{coaction}
\[
\delta : A \to A \otimes H
\] 
is an algebra map.
The subalgebra~$A^H$ of \emph{coin\-var\-iants} of an $H$-comodule algebra~$A$
is the subalgebra
\begin{equation*}
A^H = \{ a \in A \, | \, \delta(a)  = a \otimes 1\} \, .
\end{equation*}

Twisted comodule algebras are basic examples of comodule algebras; let us recall their definition.
A \emph{two-cocycle}~$\alpha$ on a Hopf algebra~$H$
is a bilinear form $\alpha : H \times H \to k$ satisfying the cocycle condition
\begin{equation*}\label{cocycle}
\alpha(x_1,y_1)\, \alpha(x_2 y_2, z)
= \alpha(y_1, z_1)\, \alpha(x, y_2 z_2)
\end{equation*}
for all $x,y,z \in H$.
We always assume that $\alpha$ is invertible (with respect to the convolution product)
and normalized, i.e., $\alpha(x,1)  = \alpha(1,x) = \varepsilon(x)$ for all $x\in H$.

Let $u_H$ be a copy of the underlying vector space of~$H$.
Denote the identity map~$u$ from $H$ to~$u_H$ by $x \mapsto u_x$ ($x\in H$).
The \emph{twisted comodule algebra} ${}^{\alpha} H$ is defined as the vector space~$u_H$ equipped 
with the product given by
\begin{equation}\label{twisted-multiplication}
u_x \,   u_y = \alpha(x_1, y_1) \, u_{x_2 y_2}
\end{equation}
for all $x$, $y \in H$.
This product is associative thanks to the cocycle condition.
As~$\alpha$ is normalized, the unit of~${}^{\alpha} H$ is~$u_1$.

The algebra ${}^{\alpha} H$ is an $H$-comodule algebra
with coaction 
$\delta \colon {}^{\alpha} H \to {}^{\alpha} H \otimes H$
given for all $x\in H$ by
\begin{equation*}\label{twisted-coaction}
\delta (u_x) = u_{x_1} \otimes x_2 \, .
\end{equation*}
It is easy to check that the subalgebra of coinvariants of~${}^{\alpha} H$ coincides with~$k \, u_1$
(for more on twisted comodule algebras, see\,\cite[Sect.\,7]{Mo}).

\subsection{$H$-identities and the universal comodule algebra}\label{ssec-PI}

Let us recall the notion of an $H$-identity for an $H$-comodule algebra~$A$,
as introduced in\,\cite[Sect.\,2.2]{AK}.

Take a copy~$X_H$ of~$H$; the identity map from~$H$ to~$X_H$
sends an element $x\in H$ to the symbol~$X_x$. 
The map~$x \mapsto X_x$ is linear and is determined by its values on a linear basis of~$H$. 
Now consider the tensor algebra on~$X_H$:
\[
T(X_H) = \bigoplus_{i\geq 0} \, T^i(X_H) \, ,
\]
where $T^i(X_H) = (X_H)^{\otimes i}$.

There is a tautological $H$-comodule algebra structure on~$T(X_H)$ 
with coaction $\delta : T(X_H) \to T(X_H) \otimes H$ given on each generator~$X_x$ by
\begin{equation*}\label{T-coaction}
\delta(X_x) = X_{x_1} \otimes x_2 \, .
\end{equation*}

\begin{definition}
Given an $H$-comodule algebra~$A$, we say that
an element $P \in T(X_H)$ is an $H$-identity for~$A$
if $\mu(P) = 0$ for all $H$-comodule algebra maps $\mu : T(X_H) \to A$.
\end{definition}

Denote the set of all $H$-identities for~$A$ by~$\Id_H(A)$.
By~\cite[Prop.~2.2]{AK} 
the set~$I_H(A)$ is a two-sided ideal, right $H$-coideal of~$T(X_H)$,
and it is preserved by all comodule algebra endomorphisms of~$T(X_H)$.

The quotient algebra
$\UU_H(A) = T(X_H)/I_H(A)$
is an $H$-comodule algebra such that the canonical surjection $T(X_H) \to \UU_H(A)$
is a comodule algebra map. By definition, all $H$-identities for $A$ vanish in~$\UU_H(A)$,
which is the biggest quotient of~$T(X_H)$ for which this happens.
We call $\UU_H(A)$ the \emph{universal $H$-comodule algebra} attached to the $H$-comodule algebra~$A$
(in the classical literature on polynomial identities\,\cite{Ro}, $\UU_H(A)$ is called the \emph{relatively free algebra}).

\subsection{The universal comodule algebra map}\label{ssec-detect}

We now recall how to detect $H$-identities 
when $A$ is a twisted comodule algebra as defined in Section\,\ref{ssec-prel}.

Let  $\alpha$ be a normalized convolution invertible two-cocycle on~$H$
and ${}^{\alpha} H$ the corresponding twisted comodule algebra.
As shown in\,\cite[Sect.\,4]{AK}, the $H$-identities for~${}^{\alpha} H$ can be detected 
by a comodule algebra map
\[
\mu_{\alpha}: T(X_H) \to S(t_H) \otimes {}^{\alpha} H \, ,
\]
whose definition we now detail.

Consider a copy $t_H$ of~$H$, identifying each $x\in H$ linearly with the symbol~$t_x \in t_H$.
Define $S(t_H)$ to be the symmetric algebra on the vector space~$t_H$.

The algebra~$S(t_H) \otimes {}^{\alpha}H$ is generated by the symbols $t_x u_y$ ($x,y \in H)$ as a $k$-algebra
(we drop the tensor product sign~$\otimes$ between the $t$-symbols and the $u$-symbols). 
It is a comodule algebra whose coaction is $S(t_H)$-linear and extends the coaction of~${}^{\alpha}H$:
\begin{equation*}
\delta(t_x u_y) = t_x u_{y_1} \otimes y_2 \, .
\end{equation*}

The algebra map $\mu_{\alpha}: T(X_H) \to S(t_H) \otimes {}^{\alpha} H$ is defined for all $x\in H$ by
\begin{equation}\label{mu}
\mu_{\alpha}(X_x) = t_{x_1} \otimes u_{x_2} \, .
\end{equation}
It is an $H$-comodule algebra map.
We call $\mu_{\alpha}$ the \emph{universal comodule algebra map}.
Its main property is the following (for a proof, see\,\cite[Th.\,4.3]{AK}). 

\begin{prop}\label{detect}
An element $P \in T(X_H)$ is an $H$-identity for~${}^{\alpha} H$ 
if and only if $\mu_{\alpha}(P) = 0$. 
In other words, 
$I_H({}^{\alpha} H)  = \Ker \mu_{\alpha}$.
\end{prop}

Let $\UU_H^{\alpha} =  \UU_H({}^{\alpha} H)$ and $I_H^{\alpha} =  I_H({}^{\alpha} H)$. 
It follows from the previous proposition that $\mu_{\alpha}$ induces an injection of comodule algebras
\begin{equation}\label{inject}
\UU_H^{\alpha} = T(X_H)/I_H^{\alpha} \, \hookrightarrow \, S(t_H) \otimes {}^{\alpha} H\, .
\end{equation}
We still denote the map\,\eqref{inject} by~$\mu_{\alpha}$.

The algebra $\UU_H^{\alpha}$ has two natural subalgebras, 
namely the algebra~$\VV_H^{\alpha} = (\UU_H^{\alpha})^H$ of coinvariants 
and the center~$\ZZ_H^{\alpha}$ of~$\UU_H^{\alpha}$.

In\,\cite[Prop.\,4.7]{AK} (see also \cite[Cor.\,3.3]{Ka1})
we proved that an element of~$\UU_H^{\alpha}$ is coinvariant
if and only if its image under~$\mu_{\alpha}$ belongs to the subalgebra $S(t_H) \otimes 1$ 
of~$S(t_H) \otimes {}^{\alpha} H$. In view of\,\eqref{inject}, $\VV_H^{\alpha}$ is a subalgebra of~$S(t_H)$:
\begin{equation*}
\VV_H^{\alpha} \, \hookrightarrow \, S(t_H) \, .
\end{equation*}
It follows that $\VV_H^{\alpha}$ is a commutative domain (i.e., without zero-divisors).

Similarly, by\,\cite[Prop.\,8.2]{AK} an element of~$\UU_H^{\alpha}$ is central
if and only if its image belongs to the subalgebra $S(t_H) \otimes Z({}^{\alpha} H)$,
where $Z({}^{\alpha} H)$ is the center of~${}^{\alpha} H$.
Again, in view of\,\eqref{inject} the center~$\ZZ_H^{\alpha}$ of~$\UU_H^{\alpha}$
sits inside $S(t_H) \otimes Z({}^{\alpha} H)$:
\begin{equation*}
\ZZ_H^{\alpha} \, \hookrightarrow \, S(t_H) \otimes Z({}^{\alpha} H) \, .
\end{equation*}
Therefore, $\ZZ_H^{\alpha}$ is a domain if $Z({}^{\alpha} H) $ has no zero-divisors.

It follows from the previous observations that the coinvariant elements of~$\UU_H^{\alpha}$ are all central:
$\VV_H^{\alpha} \subset \ZZ_H^{\alpha}$,
and that $\VV_H^{\alpha}  = \ZZ_H^{\alpha}$ if
the center of~${}^{\alpha} H$ is one-dimensional, in which case
the two-cocycle~$\alpha$ was called \emph{non-degenerate} in \cite[Sect.\,9]{AK}.

\subsection{Localizing the symmetric algebra}\label{subsec-loc}

By~\cite[Lemma~A.1]{AK}
there is a unique linear map $x \mapsto t^{-1}_x$
from~$H$ to the field of fractions~$\Frac S(t_H)$ of the symmetric algebra~$S(t_H)$
such that for all $x\in H$,
\begin{equation}\label{tt}
\sum_{(x)}\, t_{x_{1}} \, t^{-1}_{x_{2}}
= \sum_{(x)}\, t^{-1}_{x_{1}} \,  t_{x_{2}} = \eps(x) \, 1 \, .
\end{equation}
When $x$ is a \emph{group-like} element of~$H$, 
i.e., such that $\Delta(x) = x\otimes x$ and $\eps(x) = 1$,
then $t^{-1}_x = 1/t_x$.

Following~\cite[App.~B]{AK}, we denote by~$S(t_H)_{\Theta}$ the subalgebra of~$\Frac S(t_H)$ generated 
by all elements~$t_x$ and~$t^{-1}_x$ ($x\in H$). 

When $H$ is a \emph{pointed} Hopf algebra, then its coradical is the algebra of the group~$G(H)$
of group-like elements of~$H$.
In this case,  by\,\cite[Prop.\,B.2]{AK}, $S(t_H)_{\Theta}$ has a simple description
as the following localization of~$S(t_H)$:
\begin{equation}\label{pointed-loc}
S(t_H)_{\Theta} = S(t_H) \left[ \frac{1}{t_x} \right] _{x\in G(H)} 
\, .
\end{equation}

Recall also that for a general Hopf algebra~$H$ the algebra~$S(t_H)_{\Theta}$ 
carries a commutative Hopf algebra structure with coproduct~$\Delta$, counit~$\eps$
and antipode~$S$ given for all $x\in H$ by
\begin{equation}\label{Hopf-str}
\Delta(t_x) = t_{x_1} \otimes t_{x_2} \, , \quad \Delta(t^{-1}_x) = t^{-1}_{x_2} \otimes t^{-1}_{x_1} \, ,
\end{equation}
$\eps(t_x) = \eps(t^{-1}_x) = \eps(x)$ and $S(t_x) = t^{-1}_x$. This Hopf algebra is Takeuchi's
free commutative Hopf algebra on the coalgebra underlying~$H$; it satisfies the following universal property:
for any coalgebra map $f : H \to H'$ into a \emph{commutative Hopf algebra}~$H'$, there is a unique
Hopf algebra map $\tilde{f} : S(t_H)_{\Theta} \to H'$ extending~$f$, i.e., such that $\tilde{f}(t_x) = f(x)$ for all $x\in H$;
see\,\cite[Chap.\,IV]{Ta}.

Let us apply this universal property to the canonical surjection of Hopf algebras $q: H \to H_{\ab}$, where $H_{\ab}$ is the 
largest commutative Hopf algebra quotient of~$H$. The Hopf algebra map~$q$ induces a surjection of
Hopf algebras $\tilde{q} : S(t_H)_{\Theta} \to H_{\ab}$.
Using~$\tilde{q}$, we may equip $S(t_H)_{\Theta}$ with a right $H_{\ab}$-comodule algebra structure:
its coaction is the map~$\delta$ defined as the following composition:
\begin{equation}\label{coaction-SH}
\delta : S(t_H)_{\Theta} \overset{\Delta}{\longrightarrow} S(t_H)_{\Theta} \otimes S(t_H)_{\Theta} 
\overset{\id\otimes \tilde{q}}{\longrightarrow} S(t_H)_{\Theta} \otimes H_{\ab} \, .
\end{equation}

\subsection{The generic base algebra}\label{subsec-generic}

Now, to a pair $(H,\alpha)$ consisting of a Hopf algebra~$H$ and 
a normalized convolution invertible two-cocycle~$\alpha$,
we attach a bilinear map
$\sigma_{\alpha} : H \times H \to S(t_H)_{\Theta}$
with values in the previously defined algebra~$S(t_H)_{\Theta}$.
The map~$\sigma_{\alpha}$ is given for all $x,y \in H$ by
\begin{equation}\label{sigma-def}
\sigma_{\alpha}(x,y) = t_{x_1} \, t_{y_1} \, \alpha(x_2,y_2) \, t^{-1}_{x_3 y_3} \, .
\end{equation} 
The map $\sigma_{\alpha}$ is a two-cocycle of~$H$ with values in~$S(t_H)_{\Theta}$;
by construction, $\sigma_{\alpha}$~is cohomologous to~$\alpha$ over~$S(t_H)_{\Theta}$.
We call $\sigma_{\alpha}$ the \emph{generic cocycle} associated to~$\alpha$.
The cocycle $\alpha$ being invertible, so is~$\sigma_{\alpha}$,
with inverse $\sigma_{\alpha}^{-1}$ given for all $x,y \in H$ by
\begin{equation}\label{sigma^{-1}-def}
\sigma_{\alpha}^{-1}(x,y) = t_{x_1 y_1} \, \alpha^{-1}(x_2,y_2) \, t^{-1}_{x_3}  \, t^{-1}_{y_3}\, .
\end{equation}

Following~\cite[Sect.~5]{AK} and~\cite[Sect.~3]{Ka1},
we define the \emph{generic base algebra}~$\BB_H^{\alpha}$ attached to the pair $(H,\alpha)$
to be the subalgebra of~$S(t_H)_{\Theta}$ 
generated by the values of the generic cocycle~$\sigma_{\alpha}$ and 
of its inverse~$\sigma_{\alpha}^{-1}$.
Since $\BB_H^{\alpha}$ sits inside~$S(t_H)_{\Theta}$, it is a domain.

If~$H$ is \emph{finite-dimensional}, then by Theorem\,3.6 and Corollary\,3.7 of\,\cite{KM}
the following holds:

\begin{itemize}
\item[(a)]
$\BB_H^{\alpha}$ is a finitely generated smooth Noetherian domain
of Krull dimension equal to~$\dim_k H$;

\item[(b)]
$S(t_H)_{\Theta}$ is a finitely generated projective $\BB_H^{\alpha}$-module, 
from which it follows that $S(t_H)_{\Theta}$ is integral over~$\BB_H^{\alpha}$.
\end{itemize}

In the next subsection we explain why it is important to determine~$\BB_H^{\alpha}$.

\subsection{Relating the universal comodule algebra to the generic base algebra}\label{subsec-relating}

Let again $H$ be a Hopf algebra and 
$\alpha$~a normalized convolution invertible two-cocycle.
Using Lemma\,8.1 of\,\cite{AK} and following the proof of Proposition\,9.1 of\,\emph{loc.\ cit.}, 
we can prove that the map~$\mu_{\alpha}$ induces an embedding of
the subalgebra of coinvariants~$\VV_H^{\alpha}$ of~$\UU_H^{\alpha}$ into~$\BB_H^{\alpha}$:
\begin{equation*}
\VV_H^{\alpha} \, \hookrightarrow \, \BB_H^{\alpha} \, .
\end{equation*}

Following\,\cite[Sect.\,9]{AK}, we say that a two-cocycle~$\alpha$ is \emph{nice}
if $\BB_H^{\alpha}$ is a localization of~$\VV_H^{\alpha}$. 
If~$\alpha$ is nice, then $\BB_H^{\alpha} \otimes_{\VV_H^{\alpha}} \UU_H^{\alpha}$ is a central localization of
the universal comodule algebra~$\UU_H^{\alpha}$ satisfying the following two properties.

\begin{itemize}
\item[(i)]
By\,\cite[Th.\,9.3]{AK} 
the extension
$\BB_H^{\alpha} \subset \BB_H^{\alpha} \otimes_{\VV_H^{\alpha}} \UU_H^{\alpha}$ is a \emph{cleft $H$-Galois extension};
in particular, there is a comodule isomorphism (in general not an algebra isomorphism)
\begin{equation*}
\BB_H^{\alpha} \otimes_{\VV_H^{\alpha}} \UU_H^{\alpha} \cong \BB_H^{\alpha} \otimes H \, .
\end{equation*}
It follows that after a suitable central localization the universal comodule algebra~$\UU_H^{\alpha}$
becomes a free module of rank~$\dim_k(H)$ over its subalgebra of coinvariants.

\item[(ii)]
The comodule algebra $\AA_H^{\alpha}  = \BB_H^{\alpha}\otimes_{\VV_H^{\alpha}} \UU_H^{\alpha}$ is a 
``versal deformation space'' for the \emph{forms} of~${}^{\alpha} H$: 
any twisted comodule algebra~$A$ that is a form of~${}^{\alpha} H$
(i.e., such that $k'\otimes_k A \cong k'\otimes_k {}^{\alpha} H$ for some field extension $k'$ of~$k$)
is isomorphic to a comodule algebra of the form 
\begin{equation*}
\AA_H^{\alpha}/ \mm \AA_H^{\alpha} \, ,
\end{equation*}
where~$\mm$ is some maximal ideal of~$\BB_H^{\alpha}$.
Conversely, if $H$ is finite-dimensional (or cocommutative), then for any maximal ideal~$\mm$ of~$\BB_H^{\alpha}$,
the comodule algebra~$\AA_H^{\alpha}/ \mm \AA_H^{\alpha}$ is a form of~${}^{\alpha} H$.
For details and proofs, see\,\cite[Sect.\,7]{AK}) and\,\cite[Sect.\,3]{KM}.
\end{itemize}

Therefore, it is crucial for us
\begin{itemize}
\item[(1)]
to find a presentation by generators and relations of the generic base algebra~$\BB_H^{\alpha}$
(which will allow us to obtain all its maximal ideals) and 

\item[(2)]
to determine if a given two-cocycle is nice or not.
\end{itemize}

The generic base algebra~$\BB_H^{\alpha}$ is known when $H = kG$ is a group algebra
(see\,\cite{AHN}, \cite[Example\,4.3]{Ka1}, and the appendix at the end of the present paper).
A presentation by generators and relations of~$\BB_H^{\alpha}$ when $H$ is the four-dimensional Sweedler algebra 
was given in\,\cite[Cor.\,10.4]{AK}).
For all these Hopf algebras, $\BB_H^{\alpha}$ is \emph{rational}
in the sense that it is the localization of a polynomial algebra.

We also know that any two-cocycle on a group algebra or on the Sweedler
algebra is nice (see Proposition\,9.5 and Theorem\,10.3 in\,\cite{AK}). 
Similarly, any two-cocycle on a cocommutative Hopf algebra is nice (see\,\cite[Remark\,4.6]{KM}).

\subsection{Functoriality}\label{subsec-funct}

Consider pairs $(H,\alpha)$, where $H$ is a Hopf algebra and 
$\alpha$~is a normalized convolution invertible two-cocycle.
We define a map of such pairs $(H,\alpha) \to (H',\alpha')$
to be a Hopf algebra map $\varphi : H \to H'$ such that $\alpha' \circ (\varphi \times \varphi) = \alpha$.

Now define an algebra map $\varphi_T: T(X_H) \to T(X_{H'})$ by
\begin{equation*}
\varphi_T(X_x) = X_{\varphi(x)} \qquad (x\in H) \, . 
\end{equation*}
Similarly, define an algebra map $\varphi_S: S(t_H) \otimes {}^{\alpha} H \to S(t_{H'}) \otimes {}^{\alpha'} H'$ by
\begin{equation*}
\varphi_S(t_x u_y) = t_{\varphi(x)} u_{\varphi(y)}  \qquad (x,y\in H) \, .
\end{equation*}

\begin{prop}\label{prop-U-funct}
Under these hypotheses, the map $\varphi_T$ induces a comodule algebra map
$\varphi_U: \UU_H^{\alpha} \to \UU_{H'}^{\alpha'}$.
Moreover, if $\varphi_S$ is injective, then so is~$\varphi_U$.
\end{prop}

\begin{proof}
It is easy to verify that $\mu_{\alpha'} \circ  \varphi_T = \varphi_S \circ \mu_{\alpha}$
(check it on the generators~$X_x$, $x\in H$). 
The map $\varphi_T$ is a comodule algebra map since it is induced by the Hopf algebra map~$\varphi$.
Then by Proposition\,\ref{detect} there is a unique comodule algebra map 
$\varphi_U: \UU_H^{\alpha} \to \UU_{H'}^{\alpha'}$ such that 
\begin{equation*}\label{mu-varphi}
\mu_{\alpha'} \circ  \varphi_U = \varphi_S \circ \mu_{\alpha} \, .
\end{equation*}
The second statement follows from this and 
from the injectivity of~$\varphi_S$ and of~$\mu_{\alpha}$ on~$\UU_H^{\alpha}$.
\end{proof}

Now the restriction of~$\varphi_S$ to $S(t_H)$ sends $S(t_H)$ to $S(t_{H'})$.
It is clear that it extends to an algebra map
$\varphi_S: S(t_H)_{\Theta} \to S(t_{H'})_{\Theta}$
with
\[
\varphi_S(t_x) = t_{\varphi(x)} \quad\text{and}\quad 
\varphi_S(t^{-1}_x) = t^{-1}_{\varphi(x)}  \qquad (x\in H)\, .
\]

\begin{prop}\label{prop-B-funct}
The algebra map $\varphi_S$ sends $\BB_H^{\alpha}$ to~$\BB_{H'}^{\alpha'}$.
\end{prop}

\begin{proof}
This follows from the equality
$\varphi_S \circ \sigma_{\alpha}^{\pm 1} = \sigma_{\alpha'}^{\pm 1} \circ (\varphi \times \varphi)$,
which is straightforward.
\end{proof}

\subsection{Trivial cocycle}\label{ssec-trivial}

Let us now assume that the two-cocycle $\alpha$ is the \emph{trivial} two-cocycle 
\begin{equation*}
\alpha_0: (x,y) \mapsto \eps(x) \eps(y)  \qquad (x,y \in H) \, .
\end{equation*}
In this case it follows from\,\eqref{twisted-multiplication}
that ${}^{\alpha} H$ coincides as a $H$-comodule algebra with~$H$ itself, 
the coaction being the coproduct,
and that the linear isomorphism $u: H \to u_H$ is a Hopf algebra map.
This allows us to write~$x$ instead of~$u_x$ ($x\in H)$.

When $\alpha = \alpha_0$, 
we write $I_H$ for~$I_H^{\alpha}$, $\UU_H$ for~$\UU_H^{\alpha}$,
$\VV_H$ for~$\VV_H^{\alpha}$, and $\BB_H$ for~$\BB_H^{\alpha}$.
By Proposition\,\ref{detect}, $I_H$ is the kernel of the comodule algebra map $\mu_0: T(X_H) \to S(t_H) \otimes H$
given by 
\begin{equation}\label{mu-triv}
\mu_0(X_x) = t_{x_1} \otimes x_2 \, . 
\end{equation}

Also when $\alpha = \alpha_0$, we write $\sigma^{\pm 1}$ instead of~$\sigma_{\alpha}^{\pm 1}$.
In this case Formulas\,\eqref{sigma-def} and\,\eqref{sigma^{-1}-def} become
\begin{equation}\label{sigma0-def}
\sigma(x,y) = t_{x_1} \, t_{y_1} \, t^{-1}_{x_2 y_2} 
\quad\text{and}\quad
\sigma^{-1}(x,y) = t_{x_1 y_1} \, t^{-1}_{x_2}  \, t^{-1}_{y_2} \, .
\end{equation}

The case $\alpha = \alpha_0$ is important because by\,\cite[Prop.\,2.1]{KM}
we always can reduce the generic base algebra attached to a non-trivial cocycle to
the generic base algebra attached to the trivial cocycle. This works as follows: 
given a convolution invertible two-cocycle~$\alpha$ on~$H$,
define the Hopf algebra~$L = {}^{\alpha} H ^{\alpha^{-1}}$ as the coalgebra $H$
with the product
\begin{equation*}
x * y = \alpha(x_1, y_1) \, x_2 y_2 \, \alpha^{-1}(x_3, y_3)  \qquad (x,y \in H) \, .
\end{equation*}
Then inside~$S(t_H)_{\Theta}$, we have the equality
\begin{equation*}\label{B=B}
\BB_H^{\alpha} = \BB_L \, .
\end{equation*}

Since $L = H$ when $\alpha$ is a \emph{lazy} two-cocycle in the sense of~\cite{BC},
we have $\BB_H^{\alpha} = \BB_H$ for all lazy two-cocycles. 
On a \emph{cocommutative} Hopf algebra all two-cocycles are lazy, 
so that $\BB_H^{\alpha} = \BB_H$ for any two-cocycle~$\alpha$ on such a Hopf algebra.

This is why in the sequel we concentrate on Hopf algebras equipped with the trivial cocycle. 

\begin{rem}
For a general Hopf algebra the equality $\BB_H^{\alpha} = \BB_H$ does not necessarily hold. 
It is not even true that $\BB_H^{\alpha} = \BB_H^{\beta}$ if $\alpha$ and $\beta$ are cohomologous
two-cocycles.

Indeed, let $H$ be the four-dimensional Sweedler algebra; 
any two-cocycle~$\alpha$ on~$H$ is known to be cohomologous to a lazy one, say~$\beta$.
On one hand, by the previous observations, $\BB_H^{\beta} = \BB_H$; 
we will prove in Section\,\ref{ssec-BH-Taft} that $\BB_H$ is the subalgebra of~$S(t_H)_{\Theta}$ 
of elements of degree zero for some specific $\ZZZ/2$-grading.
Now, the formula 
\[
\sigma_{\alpha}(y,y) = \frac{at_y^2 + b t_1 t_y + ct_1^2}{t_1}
\]
of\,\cite[Lemma\,10.1]{AK} shows that, since $t_1 t_y$ is of degree~$\neq 0$,
$\BB_H^{\alpha}$ must be different from~$\BB_H$
for a general twisted comodule algebra ${}^{\alpha} H$ of the Sweedler
algebra (one corresponding to a non-zero parameter~$b$).
\end{rem}

\section{Taft algebras}\label{Taft-def}

The first examples we consider are Taft algebras.
In this section we shall give a presentation of the generic base algebra~$\BB_H$ and
show that the trivial two-cocycle is nice,
thus providing a positive answer to the problem posed in\,\cite[Sect.\,5.3]{Ka1}.
We shall also describe the universal comodule algebra~$\UU_H$ as a $\BB_H$-algebra
after a suitable localization.

\subsection{Definition}

Fix an integer $n\geq 2$.
We assume that the ground field~$k$ contains a primitive $n$-th root of unity~$q$.

The Taft algebra~$H = H_{n^2}$ has the following presentation as a $k$-algebra:
\[
H = k \, \langle\, x,y \,|\, x^n = 1 \, , \; yx = q xy\, , \;  y^n = 0 \,  \rangle .
\]
The set $\{x^iy^j\}_{0 \leq i,j < n}$ is a basis of the vector space~$H$, which therefore is of dimension~$n^2$. 

The algebra~$H$ is a Hopf algebra with coproduct~$\Delta$ and counit~$\eps$ defined by
\begin{equation}\label{coproduct}
\Delta(x) = x \otimes x\, , \quad  \Delta(y) = 1 \otimes y + y \otimes x\, , \quad
\eps(x) = 1\, , \quad \eps(y) = 0 \, .
\end{equation}
When $n=2$, the algebra~$H$ is the four-dimensional Sweedler algebra.

In the sequel we shall need a formula for the image of any basis element under~$\Delta$.
Recall the following standard notation:
for any integer $j\geq 0$, set
\[
[j] =\frac{q^j-1}{q-1} \, ,
\]
and $[j]! = [j][j-1]\cdots [1]$ for $j\geq 1$, while $[0]!=1$. 
Lastly, we set
\[
\left[ \begin{matrix}
j \\
0
\end{matrix}\right]
= 1
\quad\text{and}\quad
\left[ \begin{matrix}
j \\
r
\end{matrix}\right]
= \frac{[j][j-1]\cdots [j-r+1]}{[r]!}
\]
if $1 \leq r \leq j < n$.
Then
\begin{equation}\label{Taft-coproduct}
\Delta(x^i y^j) = 
\sum_{r=0}^j \, 
\left[ \begin{matrix}
j \\
r
\end{matrix}\right]
x^i y ^r \otimes x^{i+r} y^{j-r} \, . 
\end{equation}

Recall also that $H$ is a pointed Hopf algebra with $G(H) \cong \ZZZ/n$. 
The set of its group-like elements consists of the powers of~$x$.
Also observe that the commutative Hopf algebra quotient~$H_{\ab}$ of~$H$ is given by
\begin{equation*}
H_{\ab} = H/(y) = k[G(H)] \cong k[\ZZZ/n] \, .
\end{equation*}

\subsection{The generic base algebra}\label{ssec-BH-Taft}

The algebra~$S(t_H)$ can be identified with the polynomial algebra on the indeterminates $t_{x^iy ^j}$ ($0\leq i,j < n$).
By\,\eqref{pointed-loc} the localization~$S(t_H)_{\Theta}$ of~$S(t_H)$ 
introduced in\,Section\,\ref{subsec-loc} is obtained from $S(t_H)$ 
by inverting the elements $t_1, t_x, t_{x^2}, \ldots, t_{x^{n-1}}$ 
corresponding to the group-like elements of~$H$:
\begin{equation*}
S(t_H)_{\Theta} = S(t_H) \left[ \frac{1}{t_{x^i}}  \right]_{0\leq i < n} .
\end{equation*} 

As observed in Section\,\ref{subsec-loc}, $S(t_H)_{\Theta}$ is a Hopf algebra.
By\,\eqref{coaction-SH} and \eqref{Taft-coproduct}, the coproduct on an element~$t_{x^iy ^j}$
is given by
\begin{equation}\label{Taft-coproduct-tt}
\Delta(t_{x^iy ^j}) = 
\sum_{r=0}^j \, 
\left[ \begin{matrix}
j \\
r
\end{matrix}\right]
t_{x^i y ^r} \otimes t_{x^{i+r} y^{j-r}} \, . 
\end{equation}

Let us now determine the subalgebra~$\BB_H$ of~$S(t_H)_{\Theta}$. 
Consider the set~$\Gamma_0$ consisting of the following $n$ elements of~$S(t_H)$:
$t_x t_{x^{n-1}}$ and $t_{x^i}/{(t_x)^i}$ for $0\leq i < n$, $i\neq 1$;
these elements are invertible in~$S(t_H)_{\Theta}$
and we denote by~$\Gamma_0^{-1}$ the set of inverses of the elements of~$\Gamma_0$.

Let $\Gamma_1$ be the set consisting of the elements $t_{x^iy^j} t_{x^k}$, where $j\not\equiv 0$
and $i+j+k \equiv 0 \pmod n$; the cardinality of~$\Gamma_1$ is $n(n-1)$.

\begin{theorem}\label{thm-BH-Taft}
Let $n\geq 3$.

(a) The $n^2$ elements of $\Gamma_0 \cup \Gamma_1$ are algebraically independent.

(b) The generic base algebra $\BB_H$ is given by
\begin{equation*}
\BB_H = k[\Gamma_0, \Gamma_0^{-1}, \Gamma_1] \, .
\end{equation*}
\end{theorem}

\begin{proof}
(a) Arrange the elements of the set $\Gamma_0 \cup \Gamma_1$ as a vector of functions
and construct its Jacobian matrix with respect to the vector of variables
\[
(t_1,\ldots, t_{x^{n-1}},t_y,\ldots ,t_{x^{n-1}y}, \ldots, t_{y^{n-1}},\ldots ,t_{x^{n-1}y^{n-1}}) \, .
\]
Up to sign, the determinant~$D$ of this Jacobian matrix is the product of $(t_x\cdots t_{x^{n-1}})^{n-1}$
and the Jacobian determinant~$J$
of the $n\times n$ minor corresponding to the vector
$(t_1,t_{x^{n-1}}t_x, t_{x^2}/(t_x)^2, \ldots, t_{x^{n-1}}/(t_x)^{n-1})$
with respect to the first $n$~variables $(t_1,\ldots, t_{x^{n-1}})$. 
Clearly, $D$ is non-zero if and only if $J$ is non-zero.
Since
\[
J= \frac{\left(1 + (-1)^{n-1}(n-1)\right)\, t_{x^{n-1}}}{(t_x)^{(n-2)(n+1)/2}} \, ,
\]
is non-zero if $n\geq 3$, we obtain the desired algebraic independence.

(b) The natural epimorphism $G(H) \to G(H_{\ab})$ is clearly an isomorphism. Since $H$ is pointed,
it follows from\,\cite[Th.\,3.9]{KM} that $S(t_H)_{\Theta}$ is a free $\BB_H$-module. Hence, 
by\,\cite[Lemma\,3.11]{KM}, 
\begin{equation}\label{S-coinvariant}
\BB_H = S(t_H)_{\Theta}^{H_{\ab}} \, ,
\end{equation}
where $S(t_H)_{\Theta}$ is an $H_{\ab}$-comodule algebra as explained in Section\,\ref{subsec-loc} and 
$S(t_H)_{\Theta}^{H_{\ab}}$ is the left coideal subalgebra of~$S(t_H)_{\Theta}$ consisting of
all $H_{\ab}$-coinvariants.
Now $H_{\ab}$ being the algebra of the group~$G(H) \cong \ZZZ/n$, the comodule algebra structure
on~$S(t_H)_{\Theta}$ is equivalent to a $\ZZZ/n$-grading.
Equality\,\eqref{S-coinvariant} means that $\BB_H$ consists of the elements of~$S(t_H)_{\Theta}$
of degree~$0 \in \ZZZ/n$. 

Let us determine the grading on~$S(t_H)_{\Theta}$.
Given an element $t_{x^iy^j}$ ($0 \leq i,j < n$), 
it follows from\,\eqref{coaction-SH}, \eqref{Taft-coproduct-tt} and $q(y) = 0$ that 
\[
\delta(t_{x^iy^j}) = t_{x^iy^j} \otimes q(x)^{i+j} \, ,
\]
meaning that $t_{x^iy^j}$ is of degree~$i+j \in \ZZZ/n$ since $q(x)$ generates~$G_{\ab} \cong \ZZZ/n$.
Similarly, $\delta(t^{-1}_{x^i}) = t^{-1}_{x^i} \otimes q(x)^{-i}$, which means that $t^{-1}_{x^i}$ is of degree~$-i$.
From this it is clear that $k[\Gamma_0, \Gamma_0^{-1}, \Gamma_1]$ is a subalgebra of~$\BB_H$.
To conclude it remains to check that the subalgebra~$S_0$ of degree~$0$ elements of~$S(t_H)_{\Theta}$
is generated by $\Gamma_0 \cup \Gamma_0^{-1} \cup \Gamma_1$. 
This is the object of the subsequent lemma.
\end{proof}

\begin{lemma}\label{lem-S0}
The algebra~$S_0$ is generated by~$\Gamma_0 \cup \Gamma_0^{-1} \cup \Gamma_1$.
\end{lemma}

\begin{proof}
Let $R$ denote the subalgebra of $S_0$ generated by this set.
We need to show that $R = S_0$.

Let $S'_0 = k[(t_{x^i})^{\pm 1} \mid 0\leq i < n]_0$ be the subalgebra of elements  of degree~$0$
in the subalgebra of~$S(t_H)_{\Theta}$ of Laurent polynomials in the variables $(t_{x^i})^{\pm 1}$ ($0\leq i < n$). 
First consider a Laurent monomial $f  \in S'_0$. Suppose without loss of generality that
\[
f = \frac{(t_x)^p \, t_{x^{i_1}}\cdots t_{x^{i_r}}}{ t_{x^{j_1}}\cdots t_{x^{j_s}}} 
\]
with $p+i_1+\ldots + i_r = j_1+\ldots + j_s +nk$ for some integer~$k$.
Then
\[
f= (t_x)^{nk} \left( \frac{t_{x^{i_1}}}{(t_x)^{i_1}} \right)\cdots  
\left( \frac{t_{x^{i_r}}}{(t_x)^{i_r}} \right) 
\left( \frac{t_{x^{j_1}}}{(t_x)^{j_1}} \right)^{-1} \cdots  
\left( \frac{t_{x^{j_s}}}{(t_x)^{j_s}} \right)^{-1}.
\]
Then $f\in R$ in view of
\[
(t_x)^n = \left( \frac{t_{x^{n-1}}}{(t_x)^{n-1}}\right)^{-1} \left( t_x t_{x^{n-1}} \right) .
\]

Now without loss of generality consider $h = f t_{g_1}^{p_1}\cdots t_{g_r}^{p_r} \in S_0$ for some
non-zero homogeneous $f\in k[(t_{x^i})^{\pm 1} \mid 0\leq i < n]$ and 
distinct monomials $g_1,\ldots ,g_r\in \{ x^iy^j \mid 0\leq i,j < n, \,  j\neq 0 \}$ 
with $\deg(f)+\sum_{i=1}^r \,  p_i \deg(g_i) = 0$. 
Pick $t_{x^{k_1}}, \ldots, t_{x^{k_r}}$
such that $\deg(t_{g_i}t_{x^{k_i}}) = 0$. Then
\[
h= \frac{f}{(t_{x^{k_1}})^{p_1} \cdots (t_{x^{k_r}})^{p_r} }\,  (t_{g_1}t_{x^{k_1}})^{p_1}\cdots (t_{g_r}t_{x^{k_r}})^{p_r} .
\]
Since $f/((t_{x^{k_1}})^{p_1} \cdots (t_{x^{k_r}})^{p_r} )$ belongs to~$S'_0$ 
and each $t_{g_i}t_{x^{k_i}}$ belongs to~$\Gamma_1$, we have $h\in R$.
\end{proof}

\begin{rem}\label{rem-equivalence}
The referee suggested the following nice alternate proof for Theorem\,\ref{thm-BH-Taft}\,(b).
 
 From the proof of~\cite[Th.\,3.9]{KM}, it follows that $\BB_H$ can be characterized as the 
left coideal subalgebra~$B$ of~$S(t_H)_{\Theta}$ that satisfies the following mutually equivalent
conditions (see \cite[Lemma\,3.11]{KM}):
\begin{itemize}
\item[(i)]
$B = S(t_H)_{\Theta}^{H_{\ab}}$, where $S(t_H)_{\Theta}$ is an $H_{\ab}$-comodule algebra as explained above;

\item[(ii)]
$S(t_H)_{\Theta}/(B^+) \cong H_{\ab}$, where $(B^+)$ is the ideal generated by $B^+ = B \cap \Ker\eps$, 
and the set of group-like elements contained in~$B$ is closed under inverse.
\end{itemize}

By\,\eqref{S-coinvariant}, in order to prove that $\BB_H = B$, where $B= k[\Gamma_0, \Gamma_0^{-1}, \Gamma_1]$, 
it suffices to check that $B$ satisfies Property\,(ii) above. 
It easily follows from\,\eqref{Taft-coproduct-tt} that $B$ is a left coideal of~$S(t_H)_{\Theta}$.
Now $B^+$ is generated by the elements $t_x t_{x^{n-1}} - 1$, $t_{x^i}/{(t_x)^i}-1$ ($0\leq i < n$, $i\neq 1$)
and $t_{x^iy^j} t_{x^k}$ ($j\not\equiv 0$, $i+j+k \equiv 0 \pmod n$).
Consequently, in the quotient algebra $S(t_H)_{\Theta}/(B^+)$,
the elements $t_{x^i}$ are invertible, $t_{x^i} \equiv (t_x)^i$ for all~$i$, and $t_{x^iy^j} \equiv 0$, whenever $j\not\equiv 0$.
Therefore, $S(t_H)_{\Theta}/(B^+) \cong k[\ZZZ/n] \cong H_{\ab}$.
The condition on the group-like elements in~$B$ is also satisfied since these elements form a group generated 
by $\Gamma_0 \cup \Gamma_0^{-1}$, as is seen by applying to~$B$ the algebra projection
\[
S(t_H)_{\Theta} \to kG(S(t_H)_{\Theta} ) = S(t_{kG(H)})_{\Theta}
\]
given by $t_{x^i}^{\pm 1} \mapsto t_{x^i}^{\pm 1}$, $t_{x^i y^j} \mapsto 0$ ($0 \leq i < n$, $0 < j < n$); 
the latter is indeed a Hopf algebra projection which is induced from the coalgebra projection $H \to kG(H)$, $x^i y^j \mapsto \delta_{j,0}\, x^i$.
\end{rem}

\begin{rem}
When $n=2$, the elements $ t_1$, $t_x^2$, $t_xt_y$, $t_{xy}$ are algebraically independent and
by\,\cite[Cor.\,10.4]{AK} we have
\begin{equation*}
\BB_H = k[(t_1)^{\pm 1}, (t_x^2)^{\pm 1}, t_xt_y, t_{xy}] \, .
\end{equation*}
\end{rem}

\begin{prop}\label{prop-S-BH-Taft}
As a $\BB_H$-algebra, we have
\begin{equation*}
S(t_H)_{\Theta} \cong \BB_H[t]/ \left( t^n - (t_x)^n \right) \, .
\end{equation*}
\end{prop}

Therefore, $S(t_H)_{\Theta}$ is a finite \'etale (hence integral) extension of~$\BB_H$ and
it is a free $\BB_H$-module of rank~$n$.

\begin{proof}
Define a $\BB_H$-algebra map $f: \BB_H[t]/(t^n - (t_x)^n) \to S(t_H)_{\Theta}$ by sending $t$ to~$t_x$.
It is obviously well defined.
To prove the surjectivity of~$f$, it is enough to show that each generator $t_{x^i y^j}$ is in its image.
If $j=0$, since $t_{x^i}/(t_x)^i$ belongs to~$\BB_H$, we can write
\[
t_{x^i} = \frac{t_{x^i}}{(t_x)^i} \, (t_x)^i = f\left(\frac{t_{x^i}}{(t_x)^i} \, t^i \right)  .
\]
Similarly, if $j \not\equiv 0 \pmod n$, then
\begin{eqnarray*}
t_{x^iy^j} & = & (t_{x^{-i-j}} t_{x^iy^j} ) \left(\frac{t_{x^{-i-j}}}{(t_x)^{-i-j}}\right)^{-1}  (t_x)^{i+j} \\
&  = & f\left( (t_{x^{-i-j}} t_{x^iy^j}) \left(\frac{t_{x^{-i-j}}}{(t_x)^{-i-j}}\right)^{-1}  t^{i+j} \right)  .
\end{eqnarray*}

To prove the injectivity of~$f$, it suffices to check that $1, t_x, (t_x)^2, \ldots, (t_x)^{n-1}$
are linearly independent over~$\BB_H$, which is just as easy.
\end{proof}

\subsection{The universal comodule algebra}\label{ssec-U-Taft}

The tensor algebra~$T(X_H)$ is the free algebra on the indeterminates $X_{x^iy ^j}$ ($0\leq i,j < n$).
We will use the same notation for the image of~$X_{x^iy ^j}$ 
in~$\UU_H = T(X_H)/I_H$.

Recall from Section\,\ref{ssec-detect}
that the comodule algebra map $\mu_0 : T(X_H) \to S(t_H) \otimes H$ defined by\,\eqref{mu-triv}
induces an embedding of~$\UU_H$ into $S(t_H) \otimes H$ and that
an element~$P$ belongs to the the subalgebra~$\VV_H$ of coinvariants of~$\UU_H$
if and only if~$\mu_0(P) \in S(t_H) \otimes 1$.
Moreover, since the center of~$H$ is one-dimensional, 
$\VV_H$ coincides with the center of~$\UU_H$.

Now, 
$\mu_0(X_1) = t_1$, hence $X_1 \in \VV_H$, and
\begin{equation}\label{mu-X}
\mu_0(X_{x^i}) = t_{x^i}\, x^i 
\end{equation}
for $0 < i < n$.
Therefore, $\mu_0(X_{x^i}^n) = (t_{x^i})^n \, x^{in} = (t_{x^i})^n$, 
which implies that $X_{x^i}^n$ belongs to~$\VV_H$ for $0 < i < n$.

Let $\VV'_H$ (resp.\ $\UU'_H$) be the localization of~$\VV_H$ (resp.\ of~$\UU_H$) obtained by inverting the 
central elements $X_1$, $X_{x^i}^n$ for all $i=1, \ldots, n-1$.

Since the images $\mu_0(X_1)$ and $\mu_0(X_{x^i}^n)$ are invertible in~$S(t_H)_{\Theta}$, 
the embedding~$\mu_0$ extends to an embedding $\mu_0: \UU'_H \to S(t_H)_{\Theta} \otimes H$.
The subalgebra~$\VV'_H$ consists of the elements of~$\UU'_H$ whose images belong to~$S(t_H)_{\Theta} \otimes 1$.

For simplicity, from now on we identify each element of~$\UU'_H$ with its image in $S(t_H)_{\Theta} \otimes H$.

The following result states that the trivial cocycle of a Taft algebra is nice
in the sense of Section\,\ref{subsec-relating}.

\begin{prop}\label{Taft-nice}
We have $\VV'_H = \BB_H$. 
\end{prop}

We need the following lemma to prove the proposition.

\begin{lemma}\label{Taft-lem}
Let $A$ be the $\VV'_H$-subalgebra of~$\UU'_H$ generated by~$X_x$ and~$X_y$. 
For all $i,j = 0, 1, \ldots, n-1$, 
\begin{enumerate}
\item[(a)]
the elements $t_{x^iy^j} t_{x^{n-i-j}}$ belong to~$\VV^{\prime}_H$,

\item[(b)]
the elements $X_{x^iy^j}$ belong to~$A$, and

\item[(c)]
the elements $t_{x^iy^j} \, x^{i+j}$ belong to~$A$.
\end{enumerate}
\end{lemma}

\begin{proof}
We shall prove simultaneously the three assertions by induction on~$j$.

Let us start with the case $j=0$. 
Under the identification mentioned above, \eqref{mu-X} becomes $X_{x^i} = t_{x^i}\, x^i$. 
Hence,
\begin{equation*}
X_{x^i} X_{x^{n-i}} = t_{x^i}t_{x^{n-i}}\, x^i x^{n-i} = t_{x^i}t_{x^{n-i}} \, .
\end{equation*}
It follows that $t_{x^i}t_{x^{n-i}}$ belongs to~$\VV_H$.
This proves Item\,(a).

Now, 
\begin{equation*}
X_{x^i} = t_{x^i}\, x^i = \frac{t_{x^i}}{(t_x)^i} \, (t_{x})^i\, x^i = \frac{t_{x^i}}{(t_x)^i} \, X_x^i \, .
\end{equation*}
Therefore, in order to prove Items\,(b) and\,(c) it suffices to check that $
t_{x^i}/(t_x)^i$ belongs to~$\VV'_H$.
The latter follows from the equality
\begin{equation*}
\frac{t_{x^i}}{(t_x)^i} \, X_x^n = X_{x^i} X_x^{n-i} 
\end{equation*}
and the fact that $X_x^n$ is invertible in~$\VV'_H$.

Before we proceed with the induction, we make the following observations.
From $X_y = t_1 y + t_y x$ it follows that
\[
X_y X_x - X_xX_y = (q-1) \, t_1 t_x \, xy = (q-1) \, X_1 X_x \, y \, .
\] 
Hence, 
\[
(q-1) \, X_1 X_x^n \, y = X_x^{n-1} (X_y X_x - X_xX_y)  \, .
\] 
Since $X_1$ and $X_x^n$ are invertible in~$\VV'_H$, it follows that~$y$ belongs to~$A$.

Assume now that the assertions of the lemma hold until~$j-1$.
Then by\,\eqref{Taft-coproduct} we have
\[
X_{x^iy^j} = \sum_{r=0}^j \, 
\left[ \begin{matrix}
j \\
r
\end{matrix}\right]
t_{x^iy^r}  \, x^{i+r}y^{j-r}
=  t_{x^iy^j} \, x^{i+j} + \psi \, ,
\]
where 
\[
\psi = \sum_{r=0}^{j-1} \, 
\left[ \begin{matrix}
j \\
r
\end{matrix}\right]
t_{x^iy^r}  \, x^{i+r}y^{j-r} \, .
\]
By the induction hypothesis, $t_{x^iy^r} \, x^{i+r} \in A$ for $0\leq r \leq j-1$.
Since $y\in A$ as observed above, we obtain $\psi \in A$.
Therefore, $X_{x^iy^j} -\psi \in U^{\prime}_H$. 
Now,
\[
(X_{x^iy^j} -\psi) \, X_x^{n-i-j} = t_{x^iy^j} \,  t_{x^{n-i-j}} 
\in U^{\prime}_H \cap S(t_H)_{\Theta}\otimes 1 =\VV^{\prime}_H \, .
\] 
We have thus proved that $t_{x^iy^j} \, t_{x^{n-i-j}} \in \VV^{\prime}_H$. 
Moreover, 
\[
t_{x^iy^j} \, x^{i+j} = (X_{x^iy^j} -\psi) = \frac{1}{X_x^n} \, (t_{x^iy^j} \,  t_{x^{n-i-j}} ) \, X_x^{i+j} \, 
\]
which proves that $t_{x^iy^j} \, x^{i+j} $ belongs to~$A$.
Since $\psi \in A$ as already seen, we have $X_{x^iy^j}\in A$.
This proves the assertions of the lemma.
\end{proof}

\begin{proof}[Proof of Proposition~\ref{Taft-nice}]
In view of Theorem\,\ref{thm-BH-Taft} it suffices to ckeck that the sets
$\Gamma_0$, $\Gamma_0^{-1}$ and~$\Gamma_1$ introduced in Section\,\ref{ssec-BH-Taft}
are subsets of~$\VV'_H$.

The elements of~$\Gamma_1$ belong to~$\VV'_H$ by Lemma\,\ref{Taft-lem}\,(a).
In the course of the proof of that lemma, we have already checked that 
the elements of~$\Gamma_0$ belong to~$\VV'_H$.
For the elements of~$\Gamma_0^{-1}$, we proceed as follows. First, we have 
\begin{equation*}
X_x^n X_{x^{n-1}}^n \, (t_x t_{x^{n-1}})^{-1} = X_x^{n-1} X_{x^{n-1}}^{n-1} \, .
\end{equation*}
Since $X_x^n$ and $X_{x^{n-1}}^n$ are invertible in~$\VV'_H$, 
we deduce that $(t_x t_{x^{n-1}})^{-1}$ belongs to~$\VV'_H$.

Next, for $i=0, \ldots, n-1$,
\begin{equation*}
X_{x^i}^n \, \left(\frac{t_{x^i}}{(t_x)^i}\right)^{-1} = X_{x^i}^{n-1} X_x^i \, .
\end{equation*}
Since $X_{x^i}^n$ is invertible in~$\VV'_H$, we also conclude that $(t_{x^i}/(t_x)^i)^{-1}$ belongs to~$\VV'_H$.
\end{proof}

We finally determine $\UU'_H$ as a $\BB_H$-algebra.

\begin{theorem}\label{U=Taft}
There is an isomorphism of algebras
\begin{equation*}
\UU'_H \cong 
\BB_H \left\langle\, \xi, \eta \, |\, \xi^n = X_x^n \, ,\;  \eta^n = 0 \, , \;  \eta\xi - q \xi \eta = 0 \,\right\rangle .
\end{equation*}
\end{theorem}

\begin{proof}
First recall from Proposition\,\ref{Taft-nice} that $\BB_H = \VV'_H$.
Next, let $R$ be the $\VV'_H$-algebra defined by the RHS of the isomorphism in the theorem.
Set 
\begin{equation*}
f(\xi) = X_x \quad\text{and}\quad f(\eta) = X_y - \frac{t_y}{t_x} \, X_x \, .
\end{equation*}
Observe that, since the $\ZZZ/n$-degree of~$t_y/t_x$ is~$0$, this element belongs to $\BB_H = \VV'_H$.
We have $f(\xi) = t_x \, x$ and 
\begin{equation*}
f(\eta) = (t_1 \, y + t_y \, x) - \frac{t_y}{t_x} \, t_x \, x = t_1 y \, .
\end{equation*}
From this and the defining relations of~$H$, it follows that 
$f$ defines an algebra map $R \to \UU'_H$.

Let us now show that this map is injective.
It is clear that any element $\omega \in R$ can be written uniquely as
$\omega = \sum_{i,j=0}^{n-1} \, v_{i,j} \, \xi^i \eta^j$, where $v_{i,j} \in \VV'_H$.
Now, $f(\omega) = \sum_{i,j=0}^{n-1} \, v_{i,j} \, t_x^i t_1^j \, x^i y^j$. 
Since the elements $x^i y^j$ are linearly independent in~$H$, if $f(\omega) = 0$, then
$v_{i,j} \, t_x^i t_1^j  = 0$, hence $v_{i,j} = 0$ for all $i,j$, which implies~$f=0$.

On the other hand, $f$ is surjective by Lemma\,\ref{Taft-lem}\,(b) since
$\UU'_H$ is generated by the elements $X_{x^iy^j}$ ($i,j = 0, 1, \ldots, n-1$).
We have thus proved that $f$ is an isomorphism.
\end{proof}

\section{The Hopf algebras $E(n)$}\label{sec-En}

We now deal with the Hopf algebras~$E(n)$,
as defined for instance in\,\cite{BDG},\,\cite[Example\,2.2]{BC},\,\cite{PvO}.
As in the previous section, 
we shall determine the corresponding generic base algebra and universal comodule algebra,
and show that the trivial two-cocycle is nice.

\subsection{Definition}

Fix an integer $n\geq 1$. 
The algebra~$H= E(n)$ is generated by $n+1$ elements $x$, $y_1, \ldots, y_n$
subject to the relations
\begin{equation*}
x^2=1 \, , \quad y_i^2=0 \, , \quad y_ix+xy_i=0 \, ,\quad y_iy_j+y_jy_i=0 
\end{equation*}
for all $i,j= 1, \ldots, n$.
When $n=1$, then $H$ is the Sweedler algebra.

As a vector space, $H$ is of dimension~$2^{n+1}$ with the following basis.
For any subset $I \subset\{1, 2, \ldots, n\}$, set
$y_I = y_{i_1}  \cdots y_{i_r}$ if $I = \{i_1  < \cdots < i_r\}$.
By convention, $y_I = 1$ if $I = \emptyset$.
Then 
\[
\left\{y_I, xy_I \mid I \subset\{1, 2, \ldots, n\} \right\}
\] 
is a basis of~$H$.
The elements $y_1, \ldots, y_n$ generate a subalgebra that is isomorphic to an exterior (or Grassmann) algebra.

It can checked that the center $Z(H)$ of~$H$ 
has a basis consisting of all elements $y_I$ such that $|I|$ is even. 
It is of dimension~$2^{n-1}$.

The algebra~$H$ is a Hopf algebra with coproduct~$\Delta$ and counit~$\eps$ defined by
\begin{equation}\label{En-coproduct}
\Delta(x) = x \otimes x\, , \qquad  \Delta(y_i) = 1 \otimes y_i + y_i \otimes x\, ,
\end{equation}
$\eps(x) = 1$ and $\eps(y_i) = 0$ for all $i= 1, \ldots, n$.

On the basis elements, the coproduct has the following form:
\begin{equation}\label{coproduct-yI}
\Delta (y_I)= \sum_{J\subset I}\, (-1)^{m_J} \, y_J \otimes x^{|J|} \, y_{J^c_I}
\end{equation}
and 
\begin{equation}\label{coproduct-xyI}
\Delta (xy_I)= \sum_{J\subset I}\,  (-1)^{m_J} \, x y_J \otimes x^{|J|+1} \, y_{J^c_I} \, ,
\end{equation}
where $J^c_I = \{j_1 < \cdots < j_s\}$ is the complement of~$J$ in~$I$.
The exponent~$m_J$ is defined as follows: 
if $J = \{i_1  < \cdots < i_r\}$ and $J^c_I = \{j_1 < \cdots < j_s\}$,
let $m_k$ denote for each $k = 1, \ldots, r$ the cardinality of the set $\{ j_{\ell} \mid i_k > j_{\ell} \}$;
then $m_J = \sum_{k=1}^r \, m_k $.

The Hopf algebra $H$ is pointed and $G(H) \cong \ZZZ/2$, generated by~$x$.
One checks that the commutative Hopf algebra quotient~$H_{\ab}$ of~$H$ is given by
\begin{equation*}
H_{\ab} = H/(y_1, \ldots, y_n) = k[G(H)] \cong k[\ZZZ/2] \, .
\end{equation*}

\subsection{The generic base algebra}\label{ssec-BH-En}

For this Hopf algebra, $S(t_H)$ is the polynomial algebra on the indeterminates 
$t_{y_I}$ and $t_{xy_I}$, where $I$ runs over all subsets of~$\{1, \ldots, n \}$.
Since~$H$ is pointed with $1$ and~$x$ as only group-like elements, 
the localization~$S(t_H)_{\Theta}$ of~$S(t_H)$ is obtained from~$S(t_H)$ 
by inverting~$t_1$ and~$t_x$:
\begin{equation*}
S(t_H)_{\Theta} = S(t_H) \left[ \frac{1}{t_1}, \frac{1}{t_x}  \right] .
\end{equation*}

The $H_{\ab}$-comodule algebra structure on~$S(t_H)_{\Theta}$ (see Section\,\ref{subsec-loc}) induces a 
$\ZZZ/2$-grading on it. It is easy to check from\,\eqref{coproduct-yI} and\,\eqref{coproduct-xyI}
that
\[
\deg(t_{y_I}) = |I| \quad\text{and} \quad
\deg(t_{xy_I}) = |I| +1 \in \ZZZ/2
\]
for all $I \subset \{1, \ldots, n\}$.

Set $\Gamma_0 = \{t_1, (t_x)^2 \}$ and $\Gamma_0^{-1} = \{t_1^{-1}, (t_x)^{-2} \}$.
Consider the set $\Gamma_1$ consisting of the elements $t_{y_I}$, $t_xt_{xy_I}$, where $|I|$ is even $\geq 2$, 
and of the elements $t_xt_{y_I}$, $t_{xy_I}$, where $|I|$ is odd. 
Observe that these elements are all of degree~$0$.

\begin{theorem}\label{prop-BH-En}
(a) The $2^{n+1}$ elements of $\Gamma_0 \cup \Gamma_1$ are algebraically independent
and the generic base algebra $\BB_H$ is given by
\begin{equation*}
\BB_H = k[\Gamma_0, \Gamma_0^{-1}, \Gamma_1] \, .
\end{equation*}

(b) As a $\BB_H$-algebra, we have
\begin{equation*}
S(t_H)_{\Theta} \cong \BB_H[t]/ \left( t^2 - (t_x)^2 \right) \, .
\end{equation*}
\end{theorem}

For the proof one proceeds as for Theorem\,\ref{thm-BH-Taft} and Proposition\,\ref{prop-S-BH-Taft}.

\subsection{The universal comodule algebra}

The tensor algebra~$T(t_H)$ is the free algebra on the indeterminates 
$X_{y_I}, X_{xy_I}$ ($I \subset\{1, 2, \ldots, n\}$).

Let $\VV_H$ be the subalgebra of~$\UU_H$ of coinvariant elements as before;
$\VV_H$ sits in the center of~$\UU_H$, which is bigger than~$\VV_H$ when $n>1$.
Recall that  an element~$P$ belongs to~$\VV_H$ if and only if~$\mu_0(P) \in S(t_H) \otimes 1$.
Since $\mu_0(X_1) = t_1$ and $\mu_0(X_x^2) = t_x^2$, the elements $X_1$ and $X_x^2$ belong to~$\VV_H$.

Let $\VV'_H$ (resp.\ $\UU'_H$) be the localization of~$\VV_H$ (resp.\ of~$\UU_H$) 
obtained by inverting the central elements~$X_1$ and~$X_x^2$.

As in Section\,\ref{ssec-U-Taft}, we identify each element of~$\UU'_H$ with its image in 
the algebra $S(t_H)_{\Theta} \otimes H$.

\begin{prop}\label{B-En}
We have $\BB_H = \VV'_H$.
\end{prop}

This proves that the trivial cocycle on~$E(n)$ is nice.

\begin{proof}
By Theorem\,\ref{prop-BH-En}\,(b) we need to show that
$\Gamma_0$, $\Gamma_0^{-1}$, and $\Gamma_1$ are subsets of~$\VV'_H$.
This is easy for $\Gamma_0^{\pm 1}$. Indeed,
$t_1^{\pm 1} = X_1^{\pm 1}$ and $t_x^{\pm 2}=X_x^{\pm 2}$ belong to~$\VV'_H$. 
For~$\Gamma_1$ one proceeds by induction on~$|I|$ as in the proof of Lemma\,\ref{Taft-lem}.
\end{proof}

We now determine $\UU'_H$ as a $\BB_H$-algebra.
\begin{theorem}
Let $\AA_H$ be the $\BB_H$-algebra generated by $\xi, \eta_1, \ldots, \eta_n$ subject to the relations
\begin{equation*}
\xi^2 = X_x^2 \, ,\;\;  \eta_1^2 =  \cdots = \eta_n^2 = 0 \, ,
\quad \eta_i\xi + \xi\eta_i  = 0 \, , \quad \eta_i \eta_j + \eta_j \eta_i = 0 
\end{equation*}
for all $i,j= 1, \ldots, n$.
There is an isomorphism of algebras $\UU'_H \cong \AA_H$.
\end{theorem}

\begin{proof}
The proof is similar to the proof of Theorem\,\ref{U=Taft}.
Recall from Proposition\,\ref{B-En} that $\BB_H = \VV'_H$. 
We define a $\VV'_H$-algebra map $f: \AA_H \to \UU'_H$ by
$f(\xi) = X_x = t_x \, x$ and 
\begin{equation*}
f(\eta_i) = X_{y_i} - \frac{t_{y_i}}{t_x} \, X_x  = (t_1 \, y_i + t_{y_i} \, x) - \frac{t_{y_i}}{t_x} \, t_x \, x = t_1 y_i \, .
\end{equation*}
This map is well defined.
The injectivity of~$f$ is proved as in the Taft case.
The surjectivity is an immediate consequence of the following lemma. 
\end{proof}

\begin{lemma}
As a $\VV'_H$-algebra, $\UU'_H$ is generated  by $X_x, X_{y_1},\ldots , X_{y_n}$.
\end{lemma}

\begin{proof}
Let $A$ denote the subalgebra of $\UU'_H$ generated over~$\VV'_H$ 
by the set $\{ X_x, X_{y_1},\ldots , X_{y_n} \}$. 
We need to show that $X_{y_I}, X_{xy_I} \in A$ for all subsets~$I$ of~$\{1, \ldots, n\}$. 

First we check that all $y_I$, $t_{y_I}$ and $t_{xy_I}x^{|I|+1}$ belong to~$A$.
We have
\[
X_{y_i} = t_1 \, y_i + t_{y_i} \, x = X_1 \, y_i + \frac{t_{y_i}}{t_x} \, X_x \, .
\]
From this we deduce that $X_1 y_i$ belongs to~$A$. Since $X_1$ is invertible, 
all $y_i$, hence all $y_I$, belong to~$A$.

Next, observe that, if $|I|$ is even, then $t_{y_I}x^{|I|} = t_{y_I}$, which belongs to $\VV'_H$, hence to~$A$.
If $|I|$ is odd, then $t_{y_I}x^{|I|} = t_{y_I}x = (t_xt_{y_I}/t_x^2) \, X_x$, which clearly belongs to~$A$.

Finally, if $|I|$ is odd, then  $t_{xy_I}x^{|I|+1}= t_{xy_I}$, which belongs to $\VV'_H$. 
If $|I|$ is even, then $t_{xy_I}x^{|I|+1} = t_{xy_I}x = (t_xt_{xy_I}/t_x^2) \, X_x$,
which belongs to~$A$.

By \eqref{coproduct-yI} and \eqref{coproduct-xyI} we have
\[
X_{y_I} = \sum_{J\subset I} \, (-1)^{m_J} t_{y_J} \, x^{|J|}y_{J^c_I}
\quad\text{and}\quad
X_{xy_I} = \sum_{J\subset I} \, (-1)^{m_J}t_{xy_J} \, x^{|J|+1}y_{J^c_I} \, .
\]
By the above observations the expressions inside the summation symbols are in~$A$;
hence, so are $X_{y_I}$ and~$X_{xy_I}$.  Consequently, $\UU'_H=A$.
\end{proof}

\section{Monomial Hopf algebras}\label{sec-monomial}

Fix an integer $n\geq 2$. As in Section\,\ref{Taft-def} we assume that the ground field~$k$
contains a primitive $n$-th root of unity, which we denote by~$q$.
We now extend the results of Section\,\ref{Taft-def} to the following setting.

\subsection{Monomial Hopf algebras of type~$I$}\label{ssec-def-mono}

Consider a triple $(G,x,\chi)$, where
$G$ is a finite group, $x$~a central element of~$G$ of order~$n\geq 2$,
and~$\chi$ a character $G \to k^{\times}$
such that $\chi^n = 1$ and $\chi(x) = q$.

To such a triple one associates a Hopf algebra~$H$, 
which is defined as an algebra with generators the elements~$g$ of~$G$ and 
an additional generator~$y$;
the defining relations are those of the group algebra~$kG$ as well as
\begin{equation}\label{rel-mono}
y^n = 0\, , \qquad yg = \chi(g) gy
\end{equation}
for all $g\in G$. 
A basis for $H$ is formed by the elements~$gy^i$, where $g\in G$ and $0 \leq i < n$.
Thus, the dimension of $H$ is $n\, |G|$.

The algebra $H$ has a Hopf algebra structure such that the natural inclusion $\iota: kG \to H$ is a Hopf algebra map
and
\begin{equation}\label{coproduct-mono}
\Delta(y) = 1 \otimes y + y \otimes x \, , \quad \eps(y) = 0 \, , \quad S(y) = - yx^{n-1} \, . 
\end{equation}
In the literature this Hopf algebra is called a \emph{monomial Hopf algebra of type~$I$}
(see\,\cite{Bi},\,\cite[Sect.\,7]{BC},\,\cite{CHYZ}).
If $G = \ZZZ/n$  and $x$ is a generator of~$G$, then $H$ is the $n^2$-dimensional Taft algebra.

The Hopf algebra $H$ is pointed with $G(H) = G$.
One checks that the commutative Hopf algebra quotient~$H_{\ab}$ of~$H$ is given by
\begin{equation*}
H_{\ab} = H/(y) = k[G_{\ab}]  \, ,
\end{equation*}
where $G_{\ab} = G/[G,G]$ is the abelianization of~$G$.

Observe that the algebra surjection $\pi: H \to kG$ sending $y$ to~$0$ is a Hopf algebra map
such that $\pi \circ \iota = \id$.
It follows from Propositions\,\ref{prop-U-funct} and\,\ref{prop-B-funct} that 
they induce comodule algebra maps 
\begin{equation*}
\UU_{kG} \overset{\iota_U}{\longrightarrow} \UU_H \overset{\pi_U}{\longrightarrow} \UU_{kG}
\end{equation*}
and algebra maps
\begin{equation*}
\BB_{kG} \overset{\iota_S}{\longrightarrow} \BB_H \overset{\pi_S}{\longrightarrow} \BB_{kG} 
\end{equation*}
such that $\pi_U \circ \iota_U = \id$ and $\pi_S \circ \iota_S = \id$.
Therefore, $\UU_{kG}$ (resp.\ $\BB_{kG}$) splits off~$\UU_H$ (resp.\ $\BB_H$).
Similary, passing to the coinvariants, $\VV_{kG}$ splits off~$\VV_H$.

In the sequel we shall determine $\BB_H$ and~$\UU_H$ in terms of~$\BB_{kG}$ and~$\UU_{kG}$
respectively. As for the latter, we refer to the appendix.

\subsection{The generic base algebra}\label{ssec-BH-mono}

By definition, $S(t_H)$ is the polynomial algebra on the variables $t_{gy^i}$, 
where $g\in G$ and $0 \leq i < n$.
Since $H$ is pointed and the elements of~$G$ are the only group-like elements, we have
\begin{equation*}
S(t_H)_{\Theta} = k\left[t_{gy^i} \, | \, g\in G \; \text{and}\;  0 \leq i < n \right] \left[\frac{1}{t_g}\right]_{g\in G} .
\end{equation*}
The coproduct on~$S(t_H)_{\Theta}$ is determined by
\begin{equation}\label{mono-coproduct-tt}
\Delta(t_{gy ^i}) = 
\sum_{r=0}^i \, 
\left[ \begin{matrix}
i \\
r
\end{matrix}\right]
t_{g y ^r} \otimes t_{gx^r y^{i-r}}
\end{equation}
(compare with\,\eqref{Taft-coproduct-tt}).

The $H_{\ab}$-comodule algebra structure on~$S(t_H)_{\Theta}$ (see Section\,\ref{subsec-loc}) induces a 
$G_{\ab}$-grading on it. It is easy to check from\,\eqref{coaction-SH} and\,\eqref{mono-coproduct-tt} 
that 
\[
\delta(t_{gy ^i}) = t_{gy ^i} \otimes q(gx^i) \, ,
\] 
meaning that for the $G_{\ab}$-grading, $\deg(t_{gy^i}) = \overline{gx^i}$, 
where $\overline{gx^i}$ is the image of~$gx^i$ in~$G_{\ab}$.

Consider the set $\Gamma$ of cardinality~$(n-1) |G|$
consisting of the fractions $t_{gy^i}/t_{gx^i}$ where $g\in G$ and $0 < i < n$.
By the above calculation the elements of~$\Gamma$ are all of degree~$0 \in G_{\ab}$.

\begin{theorem}\label{thm-BH-mono}
The set $\Gamma$  is algebraically independent over $\BB_{kG}$ and 
\[
\BB_H = \BB_{kG} [\Gamma] \, .
\]
\end{theorem}

\begin{proof}
The algebraic independence of~$\Gamma$ is obvious.
The equality can be proved as in Remark\,\ref{rem-equivalence}:
by Property\,(ii) there, it suffices to check that the three following conditions are satisfied for~$B = \BB_{kG}[\Gamma]$:
\begin{itemize}
\item[(a)]
$B$ is a left coideal of~$S(t_H)_{\Theta}$;

\item[(b)]
$S(t_H)_{\Theta}/(B^+) \cong H_{\ab} = k[G_{\ab}]$, where $B^+ = B \cap \Ker\eps$;

\item[(c)]
the set of group-like elements contained in~$B$ is closed under inverse.
\end{itemize}

Condition\,(a) follows from\,\eqref{mono-coproduct-tt}.
For Condition\,(b) we note that $B^+$ is generated by $\BB_{kG}^+$ and $t_{gy^i}/t_{gx^i}$ 
for all $g\in G$ and $0 < i < n$.
Consequently, $t_{gy^i} \equiv 0 \pmod {B^+}$ when $0 < i < n$. 
From this one deduces $S(t_H)_{\Theta}/(B^+) \cong k[G_{\ab}]$.
Condition\,(c) is also satisfied: indeed, the group-like elements in~$B$ coincide with the group-like elements in~$\BB_{kG}$,
which by Proposition\,\ref{prop-BG} are the elements of the group~$Y_G$ described in the appendix.
\end{proof}

\subsection{The universal comodule algebra}

As observed at the end of Section~\ref{ssec-def-mono}, $\UU_{kG}$ splits off~$\UU_H$ and
$\VV_{kG}$ splits off~$\VV_H$.

There is a localization $\VV'_{kG}$ of~$\VV_{kG}$ such that $\BB_{kG} = \VV'_{kG}$
(see Section\,\ref{ssec-UG} of the appendix).
We define a localization~$\VV'_H$ of~$\VV_H$ by
\[
\VV'_H = \VV'_{kG} \otimes_{\VV_{kG}} \VV_H \, .
\]
We also define the central localization
\[
\UU'_H =  \VV'_H \otimes_{\VV_H} \UU_H = \VV'_{kG} \otimes_{\VV_{kG}} \UU_H \, .
\]

We claim the following.

\begin{theorem}
We have $\VV'_H = \BB_H$ and there is an algebra isomorphism
\[
\UU'_H = \UU'_{kG} * k[\eta] \, /\left(\eta^n = 0, \; \eta X_g - \chi(g) X_g \eta = 0 \; |\, g\in G \right) .
\]
\end{theorem}

\begin{proof}
For the first assertion, it suffices to check that the elements $t_{gy^i}/t_{gx^i}$ of the set~$\Gamma$
belong to~$\VV'_H$. 
The proof follows the same lines as the proof of Proposition\,\ref{Taft-nice}.
We leave the details to the reader. 

The proof of the second assertion is similar to the proof of Theorem\,\ref{U=Taft}.
Let $R$ be the algebra defined by the RHS of the isomorphism in the theorem.
We define an algebra map $f: R \to \UU'_H$ that is the identity on~$\UU'_{kG}$
and sends~$\eta$ to $f(\eta) = X_y - ({t_y}/{t_x}) \, X_x$.
Using the embedding $T(X_H) \subset S(t_H) \otimes H$ provided by universal comodule map~$\mu_0$,
we obtain
$f(\eta) = t_1 \, y$ as in \emph{loc.\ cit}. Thus, $f(\eta^n) = 0$ in view of\,\eqref{rel-mono}.
Similarly, 
\[
f(\eta X_g - \chi(g) X_g \eta) = t_1 t_g \, (yg - \chi(g) \, gy) = 0 \, .
\]
The rest is left to the reader. 
\end{proof}

\appendix

\section{The group algebra case}\label{app-kG}

Let $G$ be a finite group of order~$N$
and $kG$ be the corresponding group algebra; 
we consider the latter equipped with the usual Hopf algebra structure;
its coproduct and its counit are given by $\Delta(g) = g\otimes g$ and $\eps(g) = 1$ for all $g\in G$.
This Hopf algebra is cocommutative.

The commutative Hopf algebra $(kG)_{\ab}$ is the algebra of the abelianization $G_{\ab} = G/[G,G]$ of~$G$:
\[
(kG)_{\ab} = k[G_{\ab}] \, .
\]
In this appendix we recall known facts on the generic base algebra~$\BB_H$ and 
on the universal comodule algebra~$\UU_H$ when $H= kG$.

\subsection{The generic base algebra}

The symmetric algebra~$S(t_{kG})$ is the polynomial algebra on the variables~$t_g$ ($g\in G$):
\begin{equation*}
S(t_{kG}) = k \left[\, t_g \, |\, g\in G \,\right]  .
\end{equation*}

The Hopf algebra~$kG$ is pointed and its group-like elements are exactly the elements $g\in G$.
Thus by\,\eqref{pointed-loc}
the algebra $S(t_{kG})_{\Theta}$ is the Laurent polynomial algebra
\begin{equation*}
S(t_{kG})_{\Theta} = k \left[\, t_g, t^{-1}_g \, |\, g\in G \,\right]  .
\end{equation*}
In other words, $S(t_{kG})_{\Theta} = k \ZZZ^G$ is the algebra of the free abelian group~$\ZZZ^G$ 
generated by the symbols~$t_g$ ($g\in G$).

There is a surjective homomorphism $p: \ZZZ^G \to G_{\ab}$ defined by $p(t_g) = \overline{g}$, 
where $\overline{g}$ is the image of~$g$ in~$G_{\ab}$.
Let $Y_G$ be the kernel of the homomorphism~$p$:
it is a free abelian group of rank $N= |G|$ and of index~$|G_{\ab}|$ in~$\ZZZ^G$.
The group~$Y_G$ coincides with the subgroup of~$\ZZZ^G$ generated by the elements
\begin{equation*}
\sigma(g,h) = \frac{t_g t_h}{t_{gh}}
\quad\text{and}\quad
\sigma^{-1}(g,h) = \frac{t_{gh}}{t_g t_h}  \qquad (g,h \in G) \, .
\end{equation*}
Observe that $t_e = \sigma(e,e)$ belongs to~$Y_G$, 
where $e$ denotes the identity element of~$G$.

Aljadeff, Haile, and Natapov proved the following (see\,\cite[Prop.\,9 and\,14]{AHN}).

\begin{prop}\label{prop-BG}
We have
$\BB_{kG} = k Y_G$.
\end{prop}

It follows that determining~$\BB_{kG}$ is equivalent to determining~$Y_G$:
in particular, if $(y_1, \ldots, y_N)$ is a basis of~$Y_G$, then 
$\BB_{kG}$ is the Laurent polynomial algebra
\[
\BB_{kG} = k \left[\, y_1, y_1^{-1}, \ldots, y_N , y_N^{-1} \,\right] .
\]
This shows that $\BB_{kG}$ is rational in the sense of Section\,\ref{subsec-relating}.

To find a basis of~$Y_G$ it is enough to find $N$ linearly independent elements $y_1, \ldots, y_N \in Y_G$
such that the determinant of the matrix of the elements~$y_i$ expressed in terms of the basis $(t_g)_{g\in G}$
of~$\ZZZ^G$ is, up to sign, equal to~$|G_{\ab}|$.
There are standard linear algebra techniques to find such a basis.

Another description of~$\BB_{kG}$ can be given following\,\cite[Sect.\,3.3]{KM}.
By Section\,\ref{subsec-loc} above,
the algebra $S(t_{kG})_{\Theta}$ carries a natural $G_{\ab}$-grading:
if $\overline{g}$ is an element of~$G_{\ab}$, then the corresponding $\overline{g}$-component  
of~$S(t_{kG})_{\Theta}$ is spanned by the monomials $\omega = t_{g_1}^{\pm 1} \cdots t_{g_r}^{\pm 1}$ 
such that $p(\omega) = \overline{g}$.
The generic base algebra~$\BB_{kG}$ coincides with the subalgebra of $S(t_{kG})_{\Theta}$ 
consisting of the elements of degree~$\overline{e} \in G_{\ab}$,
and there is a $\BB_{kG}$-linear isomorphism
\[
\BB_{kG} \otimes k G_{\ab} \cong S(t_{kG})_{\Theta} \, .
\]

From the previous observations one can find the following description of~$S(t_{kG})_{\Theta}$ as a $\BB_{kG}$-algebra.
Write $G_{\ab}$ as the product of $k$~cyclic groups; 
thus $G_{\ab}$ is generated by $k$~elements $\overline{g_1}, \ldots, \overline{g_k}$ of order~$n_1, \ldots, n_k$, respectively.
Then there is an algebra isomorphism
\begin{equation}\label{S-over-B}
S(t_{kG})_{\Theta} \cong \BB_{kG}[t_1, \ldots, t_k]/\left( t_i^{n^i} - (t_{g_i})^{n^i} \right)_{1\leq i \leq k} \, ,
\end{equation}
where $g_i$ is a lift of~$\overline{g_i}$ in~$G$.
The $G_{\ab}$-grading of~$S(t_{kG})_{\Theta}$ induces on the right-hand side of\,\eqref{S-over-B}
a grading for which each variable~$t_i$ is homogeneous of degree~$\overline{g_i}$.

The following lemma will be useful in the computations performed below.

\begin{lemma}\label{lem-YG}
Let $G = H \ltimes K$ be a semi-direct product of groups, where $K$ acts on~$H$. 
Suppose that this action induces a trivial action on~$H_{\ab}$.
Let $\underline{b}_H$ (resp.\ $\underline{b}_K$) be a basis of~$Y_H$ (resp.\ of~$Y_K$) containing~$t_e$.
Then 
\[
\{t_e\} \cup \left(\underline{b}_H - \{t_e\} \right) 
\cup (\underline{b}_K - \{t_e\}) \cup \left\{ \frac{t_{hk} }{t_h t_k} \right\}_{h\in H - \{e\}, \, k \in K - \{e\}}
\]
is a basis of~$Y_G$.
\end{lemma}

This lemma applies to all direct products $G = H \times K$ of groups.

\begin{proof}
The above set clearly belongs to~$Y_G$ and is linearly independent. 
Its cardinality is equal to
\[
1 + (|H| - 1) + (|K| - 1) + (|H| - 1) (|K|-1) = |H| |K| = |G| = N \, .
\]
In order to conclude, it is enough to check that the matrix of this set with respect to the basis $(t_g)_{g\in G}$
has a determinant equal to~$\pm |G_{\ab}|$.

Clearly, this matrix has four diagonal square blocks corresponding to the four subsets of the assumed basis of~$Y_G$.
The non-zero entries that are not in these blocks sit above them. Hence the 
determinant we wish to compute is the product of the determinants of these four blocks.

The block of $\{t_e\}$ has determinant~$1$.
Since $[\ZZZ^H: Y_H] = |H_{\ab}|$ and $[\ZZZ^K: Y_K] = |K_{\ab}|$,
the determinant of the block corresponding to~$\underline{b}_H - \{t_e\}$ 
(resp.\ to~$\underline{b}_K - \{t_e\}$)
is equal, up to sign, to~$|H_{\ab}|$ (resp.\ to~$|K_{\ab}|$).
The last block, of size $(|H| - 1) (|K|-1) \times (|H| - 1) (|K|-1)$ has only diagonal entries, all equal to~$1$.
Therefore, the desired determinant is equal to $|H_{\ab}| \cdot |K_{\ab}|$.
We claim that the latter is equal to~$|G_{\ab}|$, which will complete the proof.

Indeed, for any group~$G$ the abelianization $G_{\ab}$ is equal to the first homology group~$H_1(G)$
of~$G$ with coefficients in~$\ZZZ$.
Since the natural projection $G \to K$ splits, the Hochschild--Serre sequence reduces to
the short exact sequence
\[
0 \to H_0(K, H_1(H)) \to H_1(G) \to H_1(K) \to 0 \, .
\]
Since $K$ acts trivially on~$H_{\ab} = H_1(H)$, we have 
$H_0(K, H_1(H)) = H_1(H)$. It follows that $|G_{\ab}|  = |H_{\ab}| \cdot |K_{\ab}|$.
\end{proof}

We now use Proposition\,\ref{prop-BG} and Lemma\,\ref{lem-YG} to give a 
full description of~$\BB_{kG}$ for certain finite groups.

\begin{exas}
(a) \emph{(Abelian groups)}
Let $G= \ZZZ/N$ be a \emph{cyclic group} with generator~$a$ and set $y_k = t_{a^k}/(t_a)^k$.
In\,\cite[Sect.\,3.3]{Ka0} we showed that
the generic base algebra $\BB_{kG}$ is the Laurent polynomial algebra
\[
\BB_{kG} = k \left[ \, y_0^{\pm 1}, y_2^{\pm 1}, y_3^{\pm 1}, \ldots, y_N^{\pm 1} \, \right] .
\]
and that there is an algebra isomorphism
$S(t_{kG})_{\Theta} \cong \BB_{kG} [t]/(t^N - (t_a)^N) $.

Lemma\,\ref{lem-YG} allows us to reduce the determination of a basis of~$Y_G$ of a finite abelian group~$G$ 
to the case of a cyclic group.  

For instance, let $G = \ZZZ/m \times \ZZZ/n$ with generators~$a,b$ of respective orders~$m$ and~$n$.
It follows from the previous computation and of Lemma\,\ref{lem-YG} that $Y_G$ has a basis consisting of the $mn$~elements
\begin{align*}
t_e \, , \; & (t_a)^m \, , \; \frac{t_{a^2}}{(t_a)^2} \, , \ldots, \frac{t_{a^{m-1}}}{(t_a)^{m-1}} \, , \\
&  (t_b)^n \, ,  \; \frac{t_{b^2}}{(t_b)^2} \, , \ldots, \frac{t_{b^{n-1}}}{(t_b)^{n-1}} \, , \\
& \frac{t_{a^ib^j}}{t_{a^i} t_{b^j}}  \quad (1 \leq i \leq m-1 , \; 1 \leq j \leq n-1) \, .
\end{align*}

(b) \emph{(Perfect groups)}
If $G$ is perfect, i.e., $G_{\ab} = 1$, then $Y_G = \ZZZ^G$, hence 
\[
\BB_{kG} = S(t_{kG})_{\Theta}  = k \left[\, t_g, t^{-1}_g \, |\, g\in G \,\right]  .
\]
This applies for instance to all simple non-abelian groups, such as the alternating group~$A_n$ when $n\geq 5$.

(c) \emph{(Symmetric groups)}
Let $n \geq 5$ and let $G = S_n$ be the symmetric group on $n$~elements. 
We can realize $G$ as a semi-direct product~$H \ltimes K$, where $H = A_n$
and $K = \ZZZ/2$ is generated by an odd permutation~$\tau$.
We can apply Lemma\,\ref{lem-YG} since $K$ acts trivially on~$(A_n)_{\ab} = \{1\}$.
As a basis for $Y_G$, we may take the elements $t_e$, $(t_{\tau})^2$, 
$t_{\sigma}$ and $t_{\sigma\tau}/t_{\sigma}t_{\tau}$, where $\sigma$ runs over the set $A_n-\{1\}$.
Moreover, by\,\eqref{S-over-B} we have an algebra isomorphism 
\[
S(t_{kG})_{\Theta} \cong \BB_{kG} [t]/(t^2 - (t_{\tau})^2) \, .
\]
\end{exas}

\subsection{The universal comodule algebra}\label{ssec-UG}

The tensor algebra~$T(X_{kG})$ is the free algebra on the symbols~$X_g$ ($g\in G$).
Let us use the same notation for the image of~$X_g$ in the universal comodule algebra~$\UU_{kG}$.

The universal comodule algebra map~$\mu_0$ identifies the element~$X_g \in \UU_{kG}$ 
with the following element of $S(t_{kG}) \otimes kG$:
\begin{equation*}
X_g = t_g \, g \qquad (g\in G) \, .
\end{equation*}
Therefore, under~$\mu_0$, the algebra $\UU_{kG}$ projects isomorphically onto the subalgebra of~$S(t_{kG}) \otimes kG$
generated by all elements~$t_g \, g$.

Under this identification, the subalgebra of coinvariants $\VV_{kG}$ of~$\UU_{kG}$ coincides with the subalgebra of~$S(t_{kG})$
spanned by $1$ and all elements of the form $t_{g_1} \cdots t_{g_r} $,
where $g_1 \cdots g_r = e$ (see\,\cite[Prop.\,6]{AHN}).

Set $P_g = X_g X_{g^{-1}}$ and $Q_{g,h} = X_g X_h X_{(gh)^{-1}}$ for $g, h\in G$. 
These are coinvariant elements of~$\UU_{kG}$ and their images in~$S(t_{kG})$ 
(under the injection~$\mu_0$) are given by
\begin{equation*}
P_g = t_g t_{g^{-1}}
\quad\text{and}\quad
Q_{g,h} = t_g t_h t_{(gh)^{-1}} \, .
\end{equation*}
They are invertible in~$S(t_{kG})_{\Theta}$. 

By\,\cite[Cor.\,4.5]{KM} the generic base algebra~$\BB_{kG}$ is generated by the elements $P_g$, $Q_{g,h}$
and their inverses.
It follows that $\BB_{kG}$ is a localization of~$\VV_{kG}$ and thus, the trivial cocycle is {nice}
in the sense of Section\,\ref{subsec-relating}.

More precisely, let $\VV'_{kG}$ be the localization of~$\VV_{kG}$ obtained by inverting the central elements
$P_g$ and~$Q_{g,h}$ ($g,h \in G$). Then
$\BB_{kG} = \VV'_{kG}$ and there is a $G$-graded $\BB_{kG}$-linear isomorphism between 
the central localization $\UU'_{kG} = \BB_{kG} \otimes_{\VV_{kG}} \UU_{kG}$
and~$\BB_{kG} \otimes kG$ (see Property\,(i) in Section\,\ref{subsec-relating}):
\begin{equation*}
\UU'_{kG} \cong \BB_{kG} \otimes kG \, .
\end{equation*}
In particular, $\UU'_{kG}$ is a free $\BB_{kG}$-module of rank~$N$.

\section*{Acknowledgment}

This work was started when the first-named author was visiting the Institut de Recherche Math\'e\-ma\-tique Avanc\'ee (IRMA)
in Strasbourg;
she gratefully acknowledges the hospitality and excellent working conditions there.
The authors warmly thank Eli Aljadeff for useful discussions on polynomial identities
and the anonymous referee for providing shorter proofs of Theorems~\ref{thm-BH-Taft},~\ref{prop-BH-En} 
and~\ref{thm-BH-mono}.

\end{document}